\pdfoutput=1
\documentclass[a4paper,10pt]{amsart}

\usepackage[utf8]{inputenc}
\usepackage{url}

\usepackage[expansion=false]{microtype}
\usepackage{amssymb,amsmath,amsopn,amsxtra,amsthm,amsfonts}
\usepackage{xr}
\usepackage[mathcal]{euscript}
\usepackage{mathalfa}

\usepackage[english, status=final, nomargin, inline]{fixme}
\fxusetheme{color}

\usepackage[pdfusetitle,unicode]{hyperref}

\newtheorem{thm}{Theorem}
\newtheorem{cor}[thm]{Corollary}
\newtheorem{lem}[thm]{Lemma}
\newtheorem{prop}[thm]{Proposition}

\theoremstyle{definition}

\theoremstyle{remark}
\newtheorem{rem}[thm]{Remark}

\renewcommand{\eqref}[1]{(\ref{#1})}

\usepackage{tikz}
\usetikzlibrary{matrix}
\usepackage{tikz-cd}

\usepackage{fixltx2e}

\usepackage{xspace}

\usepackage{xifthen}

\usepackage{xparse}

\newcommand{\nc}{\newcommand}
\nc{\renc}{\renewcommand}
\nc{\ssec}{\subsection}
\nc{\sssec}{\subsubsection}
\nc{\on}{\operatorname}
\nc{\term}[1]{#1\xspace}

\DeclareMathSymbol{A}{\mathalpha}{operators}{`A}
\DeclareMathSymbol{B}{\mathalpha}{operators}{`B}
\DeclareMathSymbol{C}{\mathalpha}{operators}{`C}
\DeclareMathSymbol{D}{\mathalpha}{operators}{`D}
\DeclareMathSymbol{E}{\mathalpha}{operators}{`E}
\DeclareMathSymbol{F}{\mathalpha}{operators}{`F}
\DeclareMathSymbol{G}{\mathalpha}{operators}{`G}
\DeclareMathSymbol{H}{\mathalpha}{operators}{`H}
\DeclareMathSymbol{I}{\mathalpha}{operators}{`I}
\DeclareMathSymbol{J}{\mathalpha}{operators}{`J}
\DeclareMathSymbol{K}{\mathalpha}{operators}{`K}
\DeclareMathSymbol{L}{\mathalpha}{operators}{`L}
\DeclareMathSymbol{M}{\mathalpha}{operators}{`M}
\DeclareMathSymbol{N}{\mathalpha}{operators}{`N}
\DeclareMathSymbol{O}{\mathalpha}{operators}{`O}
\DeclareMathSymbol{P}{\mathalpha}{operators}{`P}
\DeclareMathSymbol{Q}{\mathalpha}{operators}{`Q}
\DeclareMathSymbol{R}{\mathalpha}{operators}{`R}
\DeclareMathSymbol{S}{\mathalpha}{operators}{`S}
\DeclareMathSymbol{T}{\mathalpha}{operators}{`T}
\DeclareMathSymbol{U}{\mathalpha}{operators}{`U}
\DeclareMathSymbol{V}{\mathalpha}{operators}{`V}
\DeclareMathSymbol{W}{\mathalpha}{operators}{`W}
\DeclareMathSymbol{X}{\mathalpha}{operators}{`X}
\DeclareMathSymbol{Y}{\mathalpha}{operators}{`Y}
\DeclareMathSymbol{Z}{\mathalpha}{operators}{`Z}

\nc{\sA}{\ensuremath{\mathcal{A}}\xspace}
\nc{\sB}{\ensuremath{\mathcal{B}}\xspace}
\nc{\sC}{\ensuremath{\mathcal{C}}\xspace}
\nc{\sD}{\ensuremath{\mathcal{D}}\xspace}
\nc{\sE}{\ensuremath{\mathcal{E}}\xspace}
\nc{\sF}{\ensuremath{\mathcal{F}}\xspace}
\nc{\sG}{\ensuremath{\mathcal{G}}\xspace}
\nc{\sH}{\ensuremath{\mathcal{H}}\xspace}
\nc{\sI}{\ensuremath{\mathcal{I}}\xspace}
\nc{\sJ}{\ensuremath{\mathcal{J}}\xspace}
\nc{\sK}{\ensuremath{\mathcal{K}}\xspace}
\nc{\sL}{\ensuremath{\mathcal{L}}\xspace}
\nc{\sM}{\ensuremath{\mathcal{M}}\xspace}
\nc{\sN}{\ensuremath{\mathcal{N}}\xspace}
\nc{\sO}{\ensuremath{\mathcal{O}}\xspace}
\nc{\sP}{\ensuremath{\mathcal{P}}\xspace}
\nc{\sQ}{\ensuremath{\mathcal{Q}}\xspace}
\nc{\sR}{\ensuremath{\mathcal{R}}\xspace}
\nc{\sS}{\ensuremath{\mathcal{S}}\xspace}
\nc{\sT}{\ensuremath{\mathcal{T}}\xspace}
\nc{\sU}{\ensuremath{\mathcal{U}}\xspace}
\nc{\sV}{\ensuremath{\mathcal{V}}\xspace}
\nc{\sW}{\ensuremath{\mathcal{W}}\xspace}
\nc{\sX}{\ensuremath{\mathcal{X}}\xspace}
\nc{\sY}{\ensuremath{\mathcal{Y}}\xspace}
\nc{\sZ}{\ensuremath{\mathcal{Z}}\xspace}

\nc{\bA}{\ensuremath{\mathbf{A}}\xspace}
\nc{\bB}{\ensuremath{\mathbf{B}}\xspace}
\nc{\bC}{\ensuremath{\mathbf{C}}\xspace}
\nc{\bD}{\ensuremath{\mathbf{D}}\xspace}
\nc{\bE}{\ensuremath{\mathbf{E}}\xspace}
\nc{\bF}{\ensuremath{\mathbf{F}}\xspace}
\nc{\bG}{\ensuremath{\mathbf{G}}\xspace}
\nc{\bH}{\ensuremath{\mathbf{H}}\xspace}
\nc{\bI}{\ensuremath{\mathbf{I}}\xspace}
\nc{\bJ}{\ensuremath{\mathbf{J}}\xspace}
\nc{\bK}{\ensuremath{\mathbf{K}}\xspace}
\nc{\bL}{\ensuremath{\mathbf{L}}\xspace}
\nc{\bM}{\ensuremath{\mathbf{M}}\xspace}
\nc{\bN}{\ensuremath{\mathbf{N}}\xspace}
\nc{\bO}{\ensuremath{\mathbf{O}}\xspace}
\nc{\bP}{\ensuremath{\mathbf{P}}\xspace}
\nc{\bQ}{\ensuremath{\mathbf{Q}}\xspace}
\nc{\bR}{\ensuremath{\mathbf{R}}\xspace}
\nc{\bS}{\ensuremath{\mathbf{S}}\xspace}
\nc{\bT}{\ensuremath{\mathbf{T}}\xspace}
\nc{\bU}{\ensuremath{\mathbf{U}}\xspace}
\nc{\bV}{\ensuremath{\mathbf{V}}\xspace}
\nc{\bW}{\ensuremath{\mathbf{W}}\xspace}
\nc{\bX}{\ensuremath{\mathbf{X}}\xspace}
\nc{\bY}{\ensuremath{\mathbf{Y}}\xspace}
\nc{\bZ}{\ensuremath{\mathbf{Z}}\xspace}

\nc{\dA}{\ensuremath{\mathds{A}}\xspace}
\nc{\dB}{\ensuremath{\mathds{B}}\xspace}
\nc{\dC}{\ensuremath{\mathds{C}}\xspace}
\nc{\dD}{\ensuremath{\mathds{D}}\xspace}
\nc{\dE}{\ensuremath{\mathds{E}}\xspace}
\nc{\dF}{\ensuremath{\mathds{F}}\xspace}
\nc{\dG}{\ensuremath{\mathds{G}}\xspace}
\nc{\dH}{\ensuremath{\mathds{H}}\xspace}
\nc{\dI}{\ensuremath{\mathds{I}}\xspace}
\nc{\dJ}{\ensuremath{\mathds{J}}\xspace}
\nc{\dK}{\ensuremath{\mathds{K}}\xspace}
\nc{\dL}{\ensuremath{\mathds{L}}\xspace}
\nc{\dM}{\ensuremath{\mathds{M}}\xspace}
\nc{\dN}{\ensuremath{\mathds{N}}\xspace}
\nc{\dO}{\ensuremath{\mathds{O}}\xspace}
\nc{\dP}{\ensuremath{\mathds{P}}\xspace}
\nc{\dQ}{\ensuremath{\mathds{Q}}\xspace}
\nc{\dR}{\ensuremath{\mathds{R}}\xspace}
\nc{\dS}{\ensuremath{\mathds{S}}\xspace}
\nc{\dT}{\ensuremath{\mathds{T}}\xspace}
\nc{\dU}{\ensuremath{\mathds{U}}\xspace}
\nc{\dV}{\ensuremath{\mathds{V}}\xspace}
\nc{\dW}{\ensuremath{\mathds{W}}\xspace}
\nc{\dX}{\ensuremath{\mathds{X}}\xspace}
\nc{\dY}{\ensuremath{\mathds{Y}}\xspace}
\nc{\dZ}{\ensuremath{\mathds{Z}}\xspace}

\nc{\bbA}{\ensuremath{\mathbb{A}}\xspace}
\nc{\bbB}{\ensuremath{\mathbb{B}}\xspace}
\nc{\bbC}{\ensuremath{\mathbb{C}}\xspace}
\nc{\bbD}{\ensuremath{\mathbb{D}}\xspace}
\nc{\bbE}{\ensuremath{\mathbb{E}}\xspace}
\nc{\bbF}{\ensuremath{\mathbb{F}}\xspace}
\nc{\bbG}{\ensuremath{\mathbb{G}}\xspace}
\nc{\bbH}{\ensuremath{\mathbb{H}}\xspace}
\nc{\bbI}{\ensuremath{\mathbb{I}}\xspace}
\nc{\bbJ}{\ensuremath{\mathbb{J}}\xspace}
\nc{\bbK}{\ensuremath{\mathbb{K}}\xspace}
\nc{\bbL}{\ensuremath{\mathbb{L}}\xspace}
\nc{\bbM}{\ensuremath{\mathbb{M}}\xspace}
\nc{\bbN}{\ensuremath{\mathbb{N}}\xspace}
\nc{\bbO}{\ensuremath{\mathbb{O}}\xspace}
\nc{\bbP}{\ensuremath{\mathbb{P}}\xspace}
\nc{\bbQ}{\ensuremath{\mathbb{Q}}\xspace}
\nc{\bbR}{\ensuremath{\mathbb{R}}\xspace}
\nc{\bbS}{\ensuremath{\mathbb{S}}\xspace}
\nc{\bbT}{\ensuremath{\mathbb{T}}\xspace}
\nc{\bbU}{\ensuremath{\mathbb{U}}\xspace}
\nc{\bbV}{\ensuremath{\mathbb{V}}\xspace}
\nc{\bbW}{\ensuremath{\mathbb{W}}\xspace}
\nc{\bbX}{\ensuremath{\mathbb{X}}\xspace}
\nc{\bbY}{\ensuremath{\mathbb{Y}}\xspace}
\nc{\bbZ}{\ensuremath{\mathbb{Z}}\xspace}

\nc{\mrm}[1]{\ensuremath{\mathrm{#1}}\xspace}
\nc{\mbf}[1]{\ensuremath{\mathbf{#1}}\xspace}
\nc{\mcal}[1]{\ensuremath{\mathcal{#1}}\xspace}
\nc{\msc}[1]{\ensuremath{\mathscr{#1}}\xspace}

\renc{\bar}[1]{\overline{#1}}

\let\sectsign\S
\let\S\relax

\nc{\sub}{\subset}
\nc{\too}{\longrightarrow}
\nc{\hook}{\hookrightarrow}
\nc*{\hooklongrightarrow}{\ensuremath{\lhook\joinrel\relbar\joinrel\rightarrow}}
\nc{\hooklong}{\hooklongrightarrow}
\nc{\twoheadlongrightarrow}{\relbar\joinrel\twoheadrightarrow}
\nc{\shiso}{\approx}
\nc{\isoto}{\xrightarrow{\sim}}
\nc{\isofrom}{\xleftarrow{\sim}}
\renc{\ge}{\geqslant}
\renc{\le}{\leqslant}
\renc{\geq}{\geqslant}
\renc{\leq}{\leqslant}

\nc{\id}{\mathrm{id}}

\DeclareMathOperator{\Hom}{\mathrm{Hom}}
\nc{\uHom}{\underline{\smash{\Hom}}}
\DeclareMathOperator{\Maps}{\mathrm{Maps}}

\DeclareMathOperator{\End}{\mathrm{End}}

\nc{\PSh}{\mathrm{PSh}{}}
\nc{\uEnd}{\underline{\smash{\End}}}

\renc{\lim}{\operatorname*{lim}}
\nc{\colim}{\operatorname*{colim}}
\nc{\Cofib}{\on{Cofib}}
\nc{\Fib}{\on{Fib}}
\nc{\initial}{\varnothing}
\nc{\op}{\mathrm{op}}
\nc{\1}{\mathbf 1}

\renc{\coprod}{\sqcup}

\nc{\bDelta}{\mbf{\Delta}}
\nc{\DM}{\mbf{DM}}
\nc{\eff}{\mathrm{eff}}
\nc{\veff}{\mathrm{veff}}
\nc{\cyc}{{\mrm{cyc}}}
\nc{\corr}{{\on{corr}}}
\nc{\fet}{{\mrm{f\acute et}}}
\nc{\fsyn}{{\mrm{fsyn}}}
\nc{\syn}{{\mrm{syn}}}
\nc{\Perf}{\mbf{Perf}}
\nc{\perf}{\on{perf}}
\nc{\oblv}{\on{oblv}}
\nc{\exact}{\on{exact}}

\nc{\F}{{\on{F}}}
\nc{\clopen}{{\mrm{clopen}}}
\nc{\B}{\mrm{B}}
\nc{\D}{\mrm{D}}
\nc{\Fin}{\on{Fin}}
\nc{\Cut}{\on{Cut}}
\nc{\Cart}{\on{Cart}}
\nc{\pairs}{\mathsf{pairs}}
\nc{\Pairs}{\mathrm{Pair}}
\nc{\Trip}{\mathrm{Trip}}
\nc{\Lab}{\mathrm{Lab}}
\nc{\coCart}{\mathrm{coCart}}
\nc{\RKE}{\mathrm{RKE}}
\nc{\strict}{\mathrm{strict}}
\nc{\Emb}{\mathrm{Emb}}
\nc{\Split}{\mathrm{Split}}
\nc{\Set}{\mathrm{Set}}
\nc{\sSets}{\mathrm{sSets}}
\nc{\pb}{\mathrm{pb}}
\nc{\fib}{\mathrm{fib}}
\nc{\diff}{\mrm{diff}}
\nc{\gp}{\mrm{gp}}
\nc{\chr}{\mrm{char}}
\nc{\mgp}{\mrm{mot-gp}}
\nc{\FSyn}{\mrm{FSyn}}
\nc{\FEt}{\mrm{FEt}}
\nc{\Spc}{\mrm{Spc}}
\nc{\Ob}{\mrm{Ob}}
\nc{\Spt}{\mrm{Spt}}
\nc{\T}{\bT}
\nc{\suspinf}{\Sigma^\infty}
\nc{\h}{\mrm{h}}
\nc{\uhom}{\underline{\mathrm{Hom}}}
\nc{\umap}{\underline{\mathrm{Maps}}}
\renc{\H}{\bH}
\nc{\Einfty}{{\sE_\infty}}
\nc{\Eone}{{\sE_1}}
\nc{\Stab}{\mrm{Stab}}
\nc{\lax}{{\mrm{lax}}}
\nc{\cocart}{{\mrm{cocart}}}
\nc{\Sch}{\on{Sch}}
\nc{\Fr}{\on{Fr}}
\nc{\A}{\mathbf{A}}
\nc{\N}{\mathbf{N}}
\nc{\Z}{\mathbf{Z}}
\nc{\Q}{\mathbf{Q}}
\nc{\Oo}{\mathcal{O}} 
\nc{\red}{{\on{red}}}
\nc{\Voev}{{\on{Voev}}}
\nc{\Corr}{\mrm{Corr}}
\nc{\Span}{\mathbf{Corr}}
\nc{\Gap}{\mrm{Gap}}
\nc{\Corrfr}{\Corr^{\fr}}
\nc{\Corrvfr}{\Corr^{\Vfr}}
\nc{\Spec}{\on{Spec}}
\nc{\Sm}{\on{Sm}}
\nc{\Gm}{\mathbf{G}_{\on{m}}}
\renc{\P}{\bP}
\nc{\nis}{\mathrm{nis}}
\nc{\zar}{\mathrm{zar}}
\nc{\et}{\mathrm{\acute et}}
\nc{\all}{\mathrm{all}}
\nc{\fold}{\mathrm{fold}}
\nc{\Fun}{\mathrm{Fun}}
\nc{\Ho}{\mathrm{Ho}}
\nc{\Segal}{\mathrm{Segal}}
\nc{\Mon}{\mrm{Mon}{}}
\nc{\Ab}{\mrm{Ab}}
\nc{\Sh}{\on{Sh}}
\nc{\M}{\mrm{M}}
\nc{\Lhtp}{L_{\A^1}}
\nc{\Lmot}{L_{\mrm{mot}}}
\nc{\mot}{\mrm{mot}}
\nc{\SH}{\mbf{SH}}
\nc{\RR}{\mbf{R}}
\nc{\CC}{\mbf{C}}
\nc{\Mod}{\mbf{Mod}}
\nc{\QCoh}{\mbf{QCoh}}
\nc{\MonUnit}{\mbf{1}}
\nc{\tr}{\on{tr}}
\nc{\vop}{\mrm{vop}}
\nc{\fr}{{\on{fr}}}
\nc{\Ar}{\mrm{Ar}}
\nc{\Vfr}{\on{Vfr}}
\nc{\frdiff}{{\on{frdiff}}}
\nc{\frGys}{\on{frGys}}
\nc{\SHfr}{\SH^{\fr}}
\nc{\SHfrdiff}{\SH^{\frdiff}}
\nc{\SHfrGys}{\SH^{\frGys}}
\nc{\InftyCat}{\infty\textnormal{-}\mrm{Cat}}
\nc{\TriCat}{\mathrm{TriCat}}
\nc{\Cat}{\mathrm{1\textnormal{-}Cat}}
\nc{\Th}{\on{Th}}

\nc{\CMon}{\mrm{CMon}{}}
\nc{\MGL}{\mrm{MGL}}
\nc{\Seg}{\mrm{Seg}{}}
\nc{\Tw}{\mrm{Tw}}
\nc{\sslash}{/\mkern-6mu/}
\nc{\PrL}{\mrm{Pr}^\mrm{L}}
\nc{\PrR}{\mrm{Pr}^\mrm{R}}
\nc{\pr}{\mrm{pr}}
\nc{\efr}{\mrm{efr}}
\nc{\nfr}{\mrm{nfr}}
\nc{\dfr}{\mrm{fr}}
\nc{\tfr}{\mrm{tfr}}
\nc{\Vect}{\mrm{Vect}}
\nc{\sVect}{\mrm{sVect}}
\nc{\fix}{\mrm{fix}}
\nc{\Hilb}{\mathrm{Hilb}}
\nc{\flci}{\mathrm{flci}}
\nc{\Isom}{\mathrm{Isom}}
\nc{\GL}{\mathrm{GL}}
\nc{\fin}{\mathrm{fin}}
\nc{\KGL}{\mrm{KGL}}
\nc{\KO}{\mrm{KO}}
\nc{\ko}{\mathit{ko}}
\nc{\spi}{\underline{\pi}}

\let\phi\varphi

\let\emptyset\varnothing

\nc{\inftyCat}{\term{$\infty$-category}}
\nc{\inftyCats}{\term{$\infty$-categories}}

\nc{\inftyOneCat}{\term{$(\infty,1)$-category}}
\nc{\inftyOneCats}{\term{$(\infty,1)$-categories}}

\nc{\inftyGrpd}{\term{$\infty$-groupoid}}
\nc{\inftyGrpds}{\term{$\infty$-groupoids}}

\nc{\inftyTop}{\term{$\infty$-topos}}
\nc{\inftyTops}{\term{$\infty$-toposes}}

\nc{\inftyTwoCat}{\term{$(\infty,2)$-category}}
\nc{\inftyTwoCats}{\term{$(\infty,2)$-categories}}

\title{The localization theorem for framed motivic spaces}

\author{Marc Hoyois}
\address{Fakultät für Mathematik\\
Universität Regensburg\\
93040 Regensburg\\
Germany}
\email{\href{mailto:marc.hoyois@ur.de}{marc.hoyois@ur.de}}
\urladdr{\url{http://www.mathematik.ur.de/hoyois/}}
\thanks{The author was partially supported by NSF grant DMS-1761718}

\date{\today}

\begin{document}

\begin{abstract}
We prove the analog of the Morel–Voevodsky localization theorem for framed motivic spaces. We deduce that framed motivic spectra are equivalent to motivic spectra over arbitrary schemes, and we give a new construction of the motivic cohomology of arbitrary schemes.
\end{abstract}

\maketitle

\parskip 0pt
\tableofcontents

\parskip 0.2cm

In this article we show that the theory of framed motivic spaces introduced in \cite{EHKSY1} satisfies localization: if $i\colon Z\hook S$ is a closed immersion of schemes, $j\colon U\hook S$ is the complementary open immersion, and $\sF\in \H^\fr(S)$ is a framed motivic space over $S$, then there is a cofiber sequence
\[
j_\sharp j^*\sF \to \sF \to i_*i^*\sF
\]
(see Theorem~\ref{thm:main}).
Consequently, the theory of framed motivic spectra satisfies Ayoub's axioms \cite{Ayoub}, which implies that it admits a full-fledged formalism of six operations. Using this formalism, we show that the equivalence $\SH^\fr(S)\simeq \SH(S)$, proved in \cite{EHKSY1} for $S$ the spectrum of a perfect field, holds for any scheme $S$ (see Theorem~\ref{thm:reconstruction}). 

The $\infty$-category $\H^\fr(S)$ of framed motivic spaces consists of $\A^1$-invariant Nisnevich sheaves on the $\infty$-category $\Span^\fr(\Sm_S)$ of smooth $S$-schemes and \emph{framed correspondences}. A framed correspondence between $S$-schemes $X$ and $Y$ is a span
\begin{equation*}
  \begin{tikzcd}
     & Z \ar[swap]{ld}{f}\ar{rd} & \\
    X &   & Y
  \end{tikzcd}
\end{equation*}
over $S$, where $f$ is a finite syntomic morphism equipped with a trivialization of its cotangent complex in the $K$-theory of $Z$.
 Our result stands in contrast to the case of finite correspondences in the sense of Voevodsky, where the analog of the Morel–Voevodsky localization theorem remains unknown.
The essential ingredient in our proof is the fact that the Hilbert scheme of framed points \cite[Definition 5.1.7]{EHKSY1} is \emph{smooth}.

\section{Review of the Morel–Voevodsky localization theorem}

We start by reviewing the localization theorem of Morel and Voevodsky \cite[\sectsign 3 Theorem 2.21]{MV}. 
We refer to \cite[Appendix C]{HoyoisGLV} for the definition of the Morel–Voevodsky $\infty$-category $\H(S)$ for $S$ an arbitrary scheme.
We shall denote by $L_\nis$, $\Lhtp$, and $\Lmot$ the localization functors enforcing Nisnevich descent, $\A^1$-invariance, and both, respectively. 

\begin{thm}[Morel–Voevodsky]
	\label{thm:MV}
	Let $i\colon Z\hook S$ be a closed immersion of schemes, $j\colon U\hook S$ the complementary open immersion, and $\sF\in\H(S)$ a motivic space over $S$. Then the square
	\[
	\begin{tikzcd}
		j_\sharp j^*\sF \ar{r} \ar{d} & \sF \ar{d} \\
		j_\sharp(*) \ar{r} & i_*i^*\sF
	\end{tikzcd}
	\]
	is coCartesian in $\H(S)$.
\end{thm}

This theorem was proved in this generality in \cite[Proposition C.10]{HoyoisGLV}, but we give here a more direct proof that was alluded to in \emph{loc}.\ \emph{cit}. In the sequel, we will actually use a slightly different form of this theorem, see Corollary~\ref{cor:A1-hensel} below.

Let $i\colon Z\hook S$ be a closed immersion with open complement $j\colon U\hook S$. For an $S$-scheme $X$ and an $S$-morphism $t\colon Z\to X$, we define the presheaf
\[
\h_S(X,t)\colon \Sch_S^\op\to\Set
\]
by the Cartesian square
\[
\begin{tikzcd}
	\h_S(X,t) \ar{r}\ar{d} & * \ar{d}{t} \\
	\h_S(X) \sqcup_{\h_S(X_U)} \h_S(U) \ar{r} & i_*\h_Z(X_Z),
\end{tikzcd}
\]
where $\h_S\colon \Sch_S\to \PSh(\Sch_S)$ is the Yoneda embedding.
Explicitly:
\[
\h_S(X,t)(Y) = \begin{cases}
	\Maps_S(Y,X) \times_{\Maps_Z(Y_Z,X_Z)} \{Y_Z\to Z\xrightarrow t X_Z\} & \text{if $Y_Z\neq\emptyset$,} \\
	* & \text{if $Y_Z=\emptyset$.}
\end{cases}
\]
If $S$ is a Henselian local scheme, we have the following well-known facts:
\begin{itemize}
	\item[(a)] If $X$ is étale over $S$, then $\h_S(X,t)(S)$ is contractible.
	\item[(b)] If $X$ is smooth over $S$, then $\h_S(X,t)(S)$ is connective (i.e., not empty).
\end{itemize}
Both assertions hold by definition of $\h_S(X,t)$ if $Z=\emptyset$. Otherwise, $(S,Z)$ is an affine Henselian pair where $Z$ has a unique closed point, so we can assume $X$ affine. Assertion (a) is then a special case of \cite[Proposition 18.5.4]{EGA4-4}, and assertion (b) is a special case of \cite[Théorème I.8]{Gruson}. For general $S$, it follows immediately that:
\begin{itemize}
	\item[(a$'$)] If $X$ is étale over $S$, then $L_\nis\h_S(X,t)$ is contractible.
	\item[(b$'$)] If $X$ is smooth over $S$, then $L_\nis\h_S(X,t)$ is connective.
\end{itemize}
Assertion (b$'$) is an abstract version of Hensel's lemma in several variables. The crux of the Morel–Voevodsky localization theorem is a refinement of (b$'$) asserting that the motivic localization $\Lmot \h_S(X,t)$ is contractible. 

\begin{lem}\label{lem:A1-hensel}
	Let $f\colon X\to Y$ be a morphism of locally finitely presented $S$-schemes that is étale in a neighborhood of $t(Z)$. Then the induced map $\h_S(X,t)\to \h_S(Y,f\circ t)$ is a Nisnevich-local isomorphism.
\end{lem}

\begin{proof}
	Since the presheaves $\h_S(X,t)$ and $\h_S(Y,f\circ t)$ transform cofiltered limits of qcqs schemes into colimits \cite[Théorème 8.8.2(i)]{EGA4-3}, it suffices to show that the given map is an isomorphism on Henselian local schemes. Since $\h_S(X,t)(T)=\h_T(X_T,t_T)(T)$, we are reduced to proving that the map $\h_S(X,t)(S)\to \h_S(Y,f\circ t)(S)$ is an isomorphism when $S$ is Henselian local; we will show that its fibers are contractible. Let $X'\subset X$ be an open neighborhood of $t(Z)$ where $f$ is étale.
	 Given a section $s\colon S\to Y$ extending $f\circ t$, we have a Cartesian square
	\[
	\begin{tikzcd}
		\h_S(X'\times_YS,(t,i))(S) \ar{r} \ar{d} & * \ar{d}{s} \\
		\h_S(X,t)(S) \ar{r} & \h_S(Y,f\circ t)(S)\rlap.
	\end{tikzcd}
	\]
	 By assertion (a) above, $\h_S(X'\times_YS,(t,i))$ is contractible, as desired.
\end{proof}

\begin{thm}[The $\A^1$-Hensel lemma]
	\label{thm:A1-hensel}
	Let $S$ be a scheme, $Z\subset S$ a closed subscheme, $X$ an $S$-scheme, and $t\colon Z\to X$ an $S$-morphism.
	If $X$ is smooth over $S$, then $\Lmot \h_S(X,t)$ is contractible.
\end{thm}

\begin{proof}
	By Lemma~\ref{lem:A1-hensel}, we can replace $X$ by any open neighborhood of $t(Z)$ in $X$. Since the question is Nisnevich-local on $S$, we can assume that $S$ and $X$ are both affine.
	Since $L_\nis\h_S(X,t)$ is connective, we can further assume that there exists a section $s\colon S\to X$ extending $t$. Then there exists an $S$-morphism $f\colon X\to\bV(\sN_s)$, étale in a neighborhood of $s(S)$, such that $f\circ s$ is the zero section of the normal bundle $\bV(\sN_s)\to S$.
	Using Lemma~\ref{lem:A1-hensel} again, we are reduced to the case where $X\to S$ is a vector bundle and $t\colon Z\to X$ is the restriction of its zero section. In this case, an obvious $\A^1$-homotopy shows that $\Lhtp\h_S(X,t)$ is contractible.
\end{proof}

\begin{rem}
	The proof of Theorem~\ref{thm:A1-hensel} actually shows that $L_\nis\Lhtp L_\nis \h_S(X,t)\simeq *$.
\end{rem}

\begin{cor}
	\label{cor:A1-hensel}
	Let $i\colon Z\hook S$ be a closed immersion with open complement $j\colon U\hook S$. For every $\sF\in \PSh(\Sm_S)$, the square
	\[
	\begin{tikzcd}
		j_\sharp j^*\sF \ar{r} \ar{d} & \sF \ar{d} \\
		\sF(\emptyset)\times\h_S(U) \ar{r} & i_*i^*\sF
	\end{tikzcd}
	\]
	is motivically coCartesian, i.e., its motivic localization is coCartesian in $\H(S)$.
\end{cor}

\begin{proof}
	Since this square preserves colimits in $\sF$, we can assume that $\sF=\h_S(X)$ for some smooth $S$-scheme $X$. We must then show that the canonical map
	\[
	\h_S(X) \coprod_{\h_S(X_U)}\h_S(U) \to i_*\h_Z(X_Z)
	\]
	is a motivic equivalence in $\PSh(\Sm_S)$. In fact, we will show that it is a motivic equivalence in $\PSh(\Sch_S)$. Writing the target as a colimit of representables, it suffices to show that for every morphism $f\colon T\to S$ and every map $\h_S(T)\to i_*\h_Z(X_Z)$, corresponding to a $T$-morphism $t\colon Z_T \to X_T$, the projection
	\[
	\left(\h_S(X) \coprod_{\h_S(X_U)}\h_S(U)\right) \times_{i_*\h_Z(X_Z)} \h_S(T) \to \h_S(T)
	\]
	is a motivic equivalence. This map is the image by the functor $f_\sharp\colon \PSh(\Sch_T)\to\PSh(\Sch_S)$ of the map
	\[
	\h_T(X_T,t) \to \h_T(T)\simeq *.
	\]
	Indeed, this follows from the projection formula $f_\sharp(f^*(B)\times_{f^*(A)} C) \simeq B\times_A f_\sharp(C)$, which holds for any morphisms $B\to A$ in $\PSh(\Sch_S)$ and $C\to f^*(A)$ in $\PSh(\Sch_T)$, and the base change equivalence $f^*i_* \simeq i_{T*}f_Z^*$.
	Since $f_\sharp$ preserves motivic equivalences and $X_T$ is smooth over $T$, Theorem~\ref{thm:A1-hensel} concludes the proof.
\end{proof}

\begin{proof}[Proof of Theorem~\ref{thm:MV}]
	The functors $j_\sharp$, $j^*$, and $i^*$ between $\infty$-categories of presheaves preserve motivic equivalences, as does the functor $i_*\colon \PSh_\Sigma(\Sm_Z) \to \PSh_\Sigma(\Sm_S)$ by \cite[Proposition 2.11]{norms}. Thus, for $\sF\in\H(S)$, the given square is the motivic localization of the square of Corollary~\ref{cor:A1-hensel}.
\end{proof}

\begin{rem}
	Arguing as in the proof of Corollary~\ref{cor:main}, we can deduce from Theorem~\ref{thm:MV} that
	\[
	\H(U) \xrightarrow{j_\sharp} \H(S) \xrightarrow{i^*}\H(Z)
	\]
	is a cofiber sequence of presentable $\infty$-categories (in fact, it is also a fiber sequence).
\end{rem}

\section{The localization theorem for framed motivic spaces}

We now turn to the proof of localization for framed motivic spaces. We use the notation from \cite{EHKSY1}.

\begin{lem}\label{lem:gamma-detect}
	The forgetful functor $\gamma_*\colon \PSh_\Sigma(\Span^\fr(\Sm_S)) \to \PSh_\Sigma(\Sm_S)$ detects Nisnevich and motivic equivalences.
\end{lem}

\begin{proof}
	This follows from \cite[Proposition 3.2.14]{EHKSY1}.
\end{proof}

\begin{prop}\label{prop:finite-pushforward}
	Let $f\colon T\to S$ be an integral morphism. Then the functor \[f_*\colon \PSh_\Sigma(\Span^\fr(\Sm_T)) \to \PSh_\Sigma(\Span^\fr(\Sm_S))\] preserves Nisnevich and motivic equivalences.
\end{prop}

\begin{proof}
	By Lemma~\ref{lem:gamma-detect}, this follows from the fact that the functor $f_*\colon \PSh_\Sigma(\Sm_T) \to \PSh_\Sigma(\Sm_S)$ preserves Nisnevich and motivic equivalences \cite[Proposition 2.11]{norms}.
\end{proof}

\begin{cor}\label{cor:finite-pushforward}
	Let $f\colon T\to S$ be an integral morphism. Then the functor \[f_*\colon \H^\fr(T) \to \H^\fr(S)\] preserves colimits.
\end{cor}

\begin{proof}
	It follows from Proposition~\ref{prop:finite-pushforward} that $f_*$ preserves sifted colimits. It also preserves limits, hence finite sums since $\H^\fr(S)$ is semiadditive \cite[Proposition 3.2.10(iii)]{EHKSY1}.
\end{proof}

If $i\colon Z\hook S$ is a closed immersion, it follows from Corollary~\ref{cor:finite-pushforward} that we have an adjunction
\[
i_*: \H^\fr(Z) \rightleftarrows \H^\fr(S): i^!.
\]

\begin{thm}[Framed localization]
	\label{thm:main}
	Let $i\colon Z\hook S$ be a closed immersion with open complement $j\colon U\hook S$. 
	Then the null-sequence
	\[
	j_\sharp j^* \to \id \to i_*i^*
	\]
	of endofunctors of $\H^\fr(S)$ is a cofiber sequence.
	Dually, the null-sequence
	\[
	i_*i^! \to \id \to j_*j^*
	\]
	of endofunctors of $\H^\fr(S)$ is a fiber sequence.
\end{thm}

\begin{proof}
	It suffices to prove the first statement.
	Since all functors involved preserve colimits by Corollary~\ref{cor:finite-pushforward}, it suffices to check that the sequence is a cofiber sequence when evaluated on $\gamma^*(X_+)$ where $X$ is smooth over $S$ and affine \cite[Proposition 3.2.10(i)]{EHKSY1}. By Proposition~\ref{prop:finite-pushforward} and Lemma \ref{lem:gamma-detect}, it suffices to show that the map
	\[
	\h^\fr_S(X) / \h^\fr_S(X_U) \to i_*\h^\fr_Z(X_Z)
	\]
	in $\PSh(\Sm_S)$ is a motivic equivalence, where $\h^\fr_S(X) / \h^\fr_S(X_U)$ denotes the quotient in commutative monoids. Note that if $Y\in \Sch_S$ is \emph{connected} then
	\[
	\h^\fr_S(X_U)(Y) = \begin{cases}
		* & \text{if $Y_Z\neq\emptyset$,} \\
		\h^\fr_S(X)(Y) & \text{if $Y_Z=\emptyset$.}
	\end{cases}
	\]
	 It follows that the canonical map
	\[
	\h^\fr_S(X)\sqcup_{j_\sharp\h^\fr_U(X_U)} \h_S(U) \to \h^\fr_S(X) / \h^\fr_S(X_U)
	\]
	is an equivalence on connected essentially smooth $S$-schemes, hence it is a Zariski-local equivalence in $\PSh(\Sm_S)$.\footnote{Here, we use the fact that $\h^\fr_S(X)$ transforms cofiltered limits of qcqs schemes into colimits (since $X$ is locally finitely presented over $S$), as well as the hypercompleteness of the clopen topology on schemes.} We are thus reduced to showing that the map
	\[
	\h^\fr_S(X)\sqcup_{j_\sharp \h^\fr_U(X_U)} \h_S(U) \to i_*\h^\fr_Z(X_Z)
	\]
	is a motivic equivalence. By \cite[Corollary 2.3.27]{EHKSY1} and the non-framed version of Proposition~\ref{prop:finite-pushforward}, we can replace $\h^\fr$ by $\h^\nfr$: it suffices to show that the map
	\[
	\h^{\nfr}_S(X)\sqcup_{j_\sharp\h^{\nfr}_U(X_U)} \h_S(U) \to i_*\h^{\nfr}_Z(X_Z)
	\]
	is a motivic equivalence in $\PSh(\Sm_S)$. By \cite[Theorem 5.1.5]{EHKSY1}, the presheaf $\h^{\nfr}_S(X)$ on \emph{all} $S$-schemes is ind-representable by smooth $S$-schemes and compatible with any base change $S'\to S$. Considering $\h^{\nfr}_S(X)$ as a presheaf on smooth $S$-schemes, this implies that $i^*(\h^{\nfr}_S(X))\simeq \h^{\nfr}_Z(X_Z)$ and $j^*(\h^{\nfr}_S(X))\simeq \h^{\nfr}_U(X_U)$. Thus, the result follows from Corollary~\ref{cor:A1-hensel}.
\end{proof}

\begin{cor}\label{cor:main}
	Let $i\colon Z\hook S$ be a closed immersion with open complement $j\colon U\hook S$. 
	Then
	\[
	\H^\fr(U) \xrightarrow{j_\sharp} \H^\fr(S) \xrightarrow{i^*} \H^\fr(Z)
	\]
	is a cofiber sequence of presentable $\infty$-categories, i.e., the functor $i_*\colon \H^\fr(Z)\to\H^\fr(S)$ is fully faithful with image $(j^*)^{-1}(0)$.
\end{cor}

\begin{proof}
	Theorem~\ref{thm:main} implies that if $j^*(A)\simeq 0$ if and only if $A\simeq i_*i^*(A)$. It also implies that the unit map $i_*\to i_*i^*i_*$ is an equivalence, hence also the counit map $i_*i^*i_*\to i_*$ by the triangle identities. It remains to show that $i_*$ is conservative. This follows immediately from the fact that every smooth $Z$-scheme admits an open covering by pullbacks of smooth $S$-schemes \cite[Proposition 18.1.1]{EGA4-4}.
\end{proof}

\begin{rem}
	Similarly, the localization theorem holds for motivic spaces with finite étale transfers or with finite syntomic transfers, because the corresponding Hilbert schemes of points in $\A^n$ are smooth.
\end{rem}

The localization theorem implies as usual the closed base change property and the closed projection formula, which states that $i_*\colon \H^\fr(Z) \to \H^\fr(S)$ is an $\H^\fr(S)$-module functor, as well as $S^1$-stable and $\T$-stable versions.

In the $\T$-stable case, using the work of Ayoub \cite{Ayoub} and Cisinski–Déglise \cite{CD}, we obtain for every separated morphism of finite type $f\colon X\to Y$ an exceptional adjunction
\[
f_!: \SH^\fr(X) \rightleftarrows \SH^\fr(Y): f^!
\]
satisfying the usual properties. In particular, framed motivic spectra satisfy proper base change and the proper projection formula.

Note that the cofiber sequence of Corollary~\ref{cor:main} is not part of a recollement in the sense of \cite[Definition A.8.1]{HA}, because $i^*$ is not left exact and the pair $(i^*,j^*)$ is not conservative. These properties are however automatic in a stable setting:

\begin{cor}\label{cor:recollement}
	Let $i\colon Z\hook S$ be a closed immersion with open complement $j\colon U\hook S$. Then the following pairs of fully faithful functors are recollements:
	\begin{enumerate}
		\item $\SH^{S^1,\fr}(Z)\xrightarrow{i_*} \SH^{S^1,\fr}(S) \xleftarrow{j_*} \SH^{S^1,\fr}(U)$,
		\item $\SH^{\fr}(Z)\xrightarrow{i_*} \SH^{\fr}(S) \xleftarrow{j_*} \SH^{\fr}(U)$.
	\end{enumerate}
\end{cor}

\begin{cor}\label{cor:conservativity}
	Let $S$ be a scheme locally of finite Krull dimension. Then the following pullback functors are conservative:
	\begin{enumerate}
		\item $\SH^{S^1,\fr}(S) \to \prod_{s\in S} \SH^{S^1,\fr}(s)$,
		\item $\SH^\fr(S) \to \prod_{s\in S} \SH^\fr(s)$.
	\end{enumerate}
\end{cor}

\begin{proof}
	We can assume $S$ qcqs and we prove the claim by induction on the dimension of $S$.
	For $s\in S$, let $\iota_s\colon \Spec \sO_{S,s}\to S$ be the canonical map. Since $\iota_s$ is pro-smooth, the pullback functor \[\iota_s^*\colon \PSh_\Sigma(\Span^\fr(\Sm_S))\to \PSh_\Sigma(\Span^\fr(\Sm_{\sO_{S,s}}))\] preserves $\A^1$-invariant Nisnevich sheaves and commutes with the internal Hom from compact objects (in particular with $\Omega$ and $\Omega_\T$). Hence, for a framed motivic $S^1$-spectrum or $\T$-spectrum $E=(E_n)_{n\geq 0}$ over $S$ and a qcqs smooth $S$-scheme $X$, the Zariski stalk of $E_n^X$ at $s$ may be computed as $\iota_s^*(E)_n(X\times_S\Spec\sO_{S,s})$.
	By the hypercompleteness of the Zariski $\infty$-topos of $S$ \cite[\sectsign 3]{ClausenMathew}, equivalences between Zariski sheaves on $S$ are detected on stalks. 
	Since the family of functors $E\mapsto E_n^X(S)$ is conservative, so is the family $\iota_s^*$ for $s\in S$.
	We can therefore assume $S$ local.
	Then the result follows from Corollary~\ref{cor:recollement} and the induction hypothesis.
\end{proof}

\begin{rem}
	Corollary~\ref{cor:conservativity} is also true if $S$ is locally Noetherian of arbitrary dimension: see the proof of \cite[Proposition 3.24]{AyoubEtale}.
\end{rem}

\section{The reconstruction theorem over a general base scheme}

Next, we extend the reconstruction theorem \cite[Theorem 3.5.12]{EHKSY1} to more general base schemes. 

\begin{lem}\label{lem:gamma-BC}
	Let $f\colon T\to S$ be a morphism of schemes. Then the canonical transformation
	\[
	f^*\gamma_* \to \gamma_*f^*\colon \H^\fr(S) \to \H(T)
	\]
	is an equivalence, and similarly for $\SH^{S^1}$ and $\SH$.
\end{lem}

\begin{proof}
	The stable statements follow from the unstable one, using the fact that the functors $\gamma_*$ and $f^*$ can be computed levelwise on prespectra.
	Since $f^*$ and $\gamma_*$ preserve sifted colimits and commute with $L_\mot$ \cite[Propositions 3.2.14 and 3.2.15]{EHKSY1}, it suffices to check that the canonical map
	\[
	f^*\h^\fr_S(X) \to \h^\fr_T(X\times_ST)
	\]
	is a motivic equivalence for every $X\in\Sm_S$ affine, where we regard $\h^\fr_S(X)$ as a presheaf on $\Sm_S$. By \cite[Corollary 2.3.27]{EHKSY1}, we can replace $\h^\fr$ by $\h^\nfr$. But the map
	\[
	f^*\h^\nfr_S(X) \to \h^\nfr_T(X\times_ST)
	\]
	is an isomorphism because $\h^\nfr_S(X)$ is a smooth ind-$S$-scheme that is stable under base change \cite[Theorem 5.1.5]{EHKSY1}.
\end{proof}

\begin{lem}\label{lem:gamma-proper}
	Let $p\colon T\to S$ be a proper morphism of schemes. Then the canonical transformation
	\[
	\gamma^* p_* \to p_*\gamma^*\colon \SH(T) \to \SH^\fr(S)
	\]
	is an equivalence.
\end{lem}

\begin{proof}
	If $p$ is a closed immersion, this follows from Theorem~\ref{thm:main} and its non-framed version.
	If $p$ is smooth and proper, this follows from the ambidexterity equivalences $p_*\simeq p_\sharp\Sigma^{-\Omega_p}$.
	Together with Zariski descent, this implies the result for $p$ locally projective. The general case (which we will not use) follows by a standard use of Chow's lemma; see \cite[Proposition 2.3.11(2)]{CD} and \cite[Proposition C.13]{HoyoisGLV} for details.
\end{proof}

\begin{thm}[Reconstruction Theorem]
	\label{thm:reconstruction}
	Let $S$ be a scheme. Then the functor \[\gamma^*\colon \SH(S)\to\SH^\fr(S)\] is an equivalence of symmetric monoidal $\infty$-categories
\end{thm}

\begin{proof}
	Since the right adjoint $\gamma_*$ is conservative \cite[Proposition 3.5.2]{EHKSY1}, it suffices to show that $\gamma^*$ is fully faithful, i.e., that the unit transformation $\id\to \gamma_*\gamma^*$ is an equivalence. 
	By Zariski descent, we may assume $S$ qcqs. In this case, the $\infty$-category $\SH(S)$ is generated under colimits by the objects $\Sigma_\T^n p_*\1_X$ for $n\in\Z$ and $p\colon X\to S$ a projective morphism \cite[Lemme 2.2.23]{Ayoub}.
	By Lemma~\ref{lem:gamma-proper}, we are thus reduced to proving that $\1_S\to \gamma_*\gamma^*\1_S$ is an equivalence.
	By Lemma~\ref{lem:gamma-BC}, we can now assume that $S=\Spec\Z$.
	By the non-framed version of Corollary~\ref{cor:conservativity} and again Lemma~\ref{lem:gamma-BC}, the result follows from the cases $S=\Spec\bQ$ and $S=\Spec\bF_p$ for $p$ prime, which are known by \cite[Theorem 3.5.12]{EHKSY1}.
\end{proof}

\begin{rem}
	The argument used in the proof of Theorem~\ref{thm:reconstruction} can be axiomatized as follows. Let $S$ be a qcqs scheme of finite Krull dimension, let
	\[\bA,\bB\colon (\Sch_S^\mathrm{qcqs})^{\op} \to \InftyCat^\mathrm{st}\]
	be functors satisfying Ayoub's axioms \cite[\sectsign 1.4.1]{Ayoub}, and let $\phi\colon \bA\to \bB$ be a natural transformation that commutes with $f_\sharp$ for $f$ smooth. Suppose that:
	\begin{enumerate}
		\item Each $\bA(X)$ is cocomplete and generated under colimits by objects of the form $f_\sharp f^*p^* (A)$ where $f\colon Y \to X$ is smooth, $p\colon X\to S$ is the structure map, and $A\in\bA(S)$.
		\item $\phi$ has a right adjoint that preserves colimits and commutes with $f^*$ for any $f$.
		\item $\phi_s\colon \bA(s)\to \bB(s)$ is fully faithful for every $s\in S$.
	\end{enumerate}
	Then $\phi_X\colon \bA(X)\to \bB(X)$ is fully faithful for every $X\in\Sch_S^\mathrm{qcqs}$.
\end{rem}

Since $\SH^\fr(S)\simeq \SH(S)\otimes_{\H(S)}\H^\fr(S)$, the reconstruction theorem implies that the right-lax symmetric monoidal functor $\Omega^\infty_\T\colon \SH(S)\to \H(S)$ factors \emph{uniquely} as
\[
\begin{tikzcd}
	\SH(S) \ar[swap]{d}{\Omega^\infty_\T} \ar[dashed]{dr}{!} & \\
	\H(S) & \H^\fr(S). \ar{l}{\gamma_*}
\end{tikzcd}
\]
Indeed, the $\infty$-groupoid of such factorizations is equivalent to that of colimit-preserving symmetric monoidal retractions of the functor $\gamma^*\colon \SH(S)\to \SH^\fr(S)$.
 In particular, the underlying cohomology theory $\Sm_S^\op \to \Spc$ of a motivic spectrum extends canonically to the $\infty$-category $\Span^\fr(\Sm_S)^\op$. As proved in \cite[Theorem 3.3.10]{EHKSY2}, this enhanced functoriality of cohomology theories can be described using the Gysin morphisms constructed using Verdier's deformation to the normal cone (see \cite{DJK}).

\section{Application to motivic cohomology}

In this final section, we obtain a simple presentation of the motivic cohomology spectrum in terms of framed correspondences. Let us denote by $H\Z_S\in \SH(S)$ Spitzweck's motivic cohomology spectrum over a base scheme $S$ \cite{SpitzweckHZ}. By construction, it is stable under arbitrary base change, and when $S$ is a Dedekind domain it represents Bloch–Levine motivic cohomology. More precisely, for such $S$, the presheaf $X\mapsto \Maps_{\SH(S)}(\Sigma^\infty_{\T}X_+,\Sigma^n_\T H\Z_S)$ on smooth $S$-schemes is the Zariski sheafification of Bloch's cycle complex $X\mapsto z^n(X,*)$ (which is known to already be a Zariski sheaf when $S$ is semilocal \cite[Theorem 1.7]{LevineLocalization}). 
When $S$ is the spectrum of a field, $H\Z_S$ is equivalent to Voevodsky's motivic cohomology spectrum.

\def\flf{\mathrm{f{}lf}}
For any commutative monoid $A$, the constant sheaf $A_S$ on $\Sm_S$ admits a canonical extension to $\Span^\flf(\Sm_S)$, where ``$\flf$'' denotes the class of finite locally free morphisms: to a span
\[
  \begin{tikzcd}
     & Z \ar[swap]{ld}{f}\ar{rd}{g} & \\
    X &   & Y
  \end{tikzcd}
\]
with $f$ finite locally free and a locally constant function $a\colon Y\to A$, we associate the locally constant function
\[
X\to A,\quad x\mapsto \sum_{z\in f^{-1}(x)} \deg_z(f)\cdot a(g(z))
\]
(see \cite[Lemma 13.13]{norms}).
In particular, $A_S$ can be regarded as an object of $\H^\fr(S)$ via the forgetful functor $\Span^\fr(\Sm_S) \to \Span^\flf(\Sm_S)$. 

If $f\colon T\to S$ is a morphism, there is an obvious map $A_S \to f_* A_T$ in $\H^\fr(S)$, whence by adjunction a map $f^* A_S \to A_T$ in $\H^\fr(T)$.

\begin{lem}\label{lem:Z-bc}
	Let $A$ be a commutative monoid and $f\colon T\to S$ a morphism of schemes. Then the canonical map $f^*A_S \to A_T$ in $\H^\fr(T)$ is an equivalence.
\end{lem}

\begin{proof}
	We consider the following commutative triangle in $\PSh_\Sigma(\Sm_T)$:
	\[
	\begin{tikzcd}
		f^*\gamma_* A_S \ar{d} \ar{dr} & \\
		\gamma_* f^* A_S \ar{r} & \gamma_* A_T.
	\end{tikzcd}
	\]
	The vertical map is a motivic equivalence by Lemma~\ref{lem:gamma-BC}, and the diagonal map is trivially a Zariski equivalence. Hence, the lower horizontal map is a motivic equivalence.
	Since $\gamma_*$ detects motivic equivalences (Lemma~\ref{lem:gamma-detect}), we are done.
\end{proof}

\begin{thm}\label{thm:HZ}
	Let $S$ be a scheme. Then there is an equivalence of motivic $\Einfty$-ring spectra $H\Z_S \simeq \gamma_* \Sigma^\infty_{\T,\fr}\Z_S$.
\end{thm}

\begin{proof}
	By Lemmas \ref{lem:gamma-BC} and \ref{lem:Z-bc}, it suffices to prove this when $S$ is a Dedekind domain. In this case, there is an isomorphism of presheaves of commutative rings
	\[
	\Omega^\infty_\T H\Z_S \simeq \Z_S.
	\]
	We claim that this isomorphism is compatible with the framed transfers on either side, the ones on the left coming from Theorem~\ref{thm:reconstruction}.
	Since we are dealing with discrete constant sheaves, it suffices to compare the transfers for a framed correspondence of the form $\eta \leftarrow T\rightarrow \eta$ where $\eta$ is a generic point of a smooth $S$-scheme. Thus we may assume that $S$ is a field, in which case we can compute the framed transfers on $\Omega^\infty_\T H\Z_S$ using \cite[Proposition 5.3.6]{EHKSY1}, verifying the claim.
	 
	 By adjunction, we obtain a morphism of $\Einfty$-algebras
	\[
	\phi_S\colon \Sigma^\infty_{\T,\fr} \Z_S \to \gamma^*H\Z_S
	\]
	in $\SH^\fr(S)$. We show that $\phi_S$ is an equivalence. By construction, $\phi_S$ is functorial in $S$. By Corollary~\ref{cor:conservativity}(2), we may therefore assume that $S$ is the spectrum of a perfect field. In this case, the recognition principle \cite[Theorem 3.5.14(i)]{EHKSY1} implies that $\phi_S$ exhibits $\gamma_* \Sigma^\infty_{\T,\fr}\Z_S$ as the very effective cover of $H\Z_S$. Since $H\Z_S$ is already very effective \cite[Lemma 13.7]{norms}, we conclude that $\phi_S$ is an equivalence.
\end{proof}

If $S$ is a Dedekind domain, the motivic spectrum $H\Z_S\in\SH(S)$ lies in the heart of the effective homotopy $t$-structure \cite[Lemma 13.7]{norms}. It follows that it admits a unique structure of strictly commutative monoid in $\SH(S)$, in the sense of \cite[\sectsign 7]{HoyoisKunneth}. Hence, for any scheme $S$, $H\Z_S\in\SH(S)$ is a module over the Eilenberg–Mac Lane spectrum $\Z\in\Spt$. In particular, for any $A\in \Mod_{\Z}(\Spt)$, we can form the tensor product $HA_S=H\Z_S\otimes_{\Z} A$. This construction defines a symmetric monoidal functor
\[
\Mod_{\Z}(\Spt) \to \Mod_{H\Z_S}(\SH(S)), \quad A\mapsto HA_S.
\]
When $S$ is the spectrum of a field and $A$ is an abelian group, $HA_S$ is equivalent to Voevodsky's motivic Eilenberg–Mac Lane spectrum with coefficients in $A$.

\begin{cor}
	Let $S$ be a scheme and $A$ an abelian group (resp.\ a ring; a commutative ring). Then there is a canonical equivalence of $H\Z_S$-modules (resp.\ of $\sA_\infty$-$H\Z_S$-algebras; of $\Einfty$-$H\Z_S$-algebras) $HA_S \simeq \gamma_* \Sigma^\infty_{\T,\fr}A_S$.
\end{cor}

\begin{proof}
	By Lemmas \ref{lem:gamma-BC} and \ref{lem:Z-bc}, we may assume that $S$ is a Dedekind domain. Since the equivalence of Theorem~\ref{thm:HZ} takes place in the heart of the effective homotopy $t$-structure, it can be uniquely promoted to an equivalence of $\Einfty$-rings in strictly commutative monoids. Hence, for any $A\in \Mod_{\Z}(\Spt_{\geq 0})$, we obtain an equivalence $HA_S\simeq \gamma_* \Sigma^\infty_{\T,\fr}(\Z_S\otimes_{\Z}A)$. To conclude, note that $\Z_S\otimes_{\Z}A\simeq A_S$ when $A$ is discrete.
\end{proof}


\bibliographystyle{alphamod}

\let\mathbb=\mathbf

{\small
\bibliography{references}
}

\parskip 0pt

\end{document}